\documentclass[12pt]{amsart}
\usepackage[utf8]{inputenc}
\usepackage{amssymb}
\usepackage{geometry}
\usepackage{amsmath,amscd}
\usepackage{graphicx}
\usepackage{color,hyperref}
\usepackage{tikz}
\usepackage{tikz-cd}
\usepackage{amsfonts}
\usetikzlibrary{matrix,arrows}
\usepackage{mathrsfs}
\usepackage{amssymb}
\usepackage{adjustbox}
\usepackage{physics}
\begin{document}
	
	\newcommand{\mf}{\mathfrak}
	\newcommand{\mc}{\mathcal}
	\newcommand{\mb}{\mathbf}
	
	% Main Rings
	\newcommand{\R}{\mathbf R}
	\newcommand{\C}{\mathbf C}
	\newcommand{\Q}{\mathbf Q}
	\newcommand{\Z}{\mathbf Z}
	\newcommand{\F}{\mathbf F}
	\newcommand{\N}{\mathbf N}

	% Operators
	\newcommand{\Fix}{\textnormal{Fix}}
	\newcommand{\End}{\textnormal{End}}
	\newcommand{\Frob}{\textnormal{Frob}}

	% Small Matrix
	\newcommand{\sm}[4]{{
			\left(\begin{smallmatrix}{#1}&{#2}\\{#3}&{#4}\end{smallmatrix}\right)}}
	\newcommand{\sv}[2]{
		\genfrac(){0pt}{1}{#1}{#2}}

	% Numbering
	\numberwithin{equation}{section}
	
	\theoremstyle{plain}
	\newtheorem{theorem}{Theorem}
	\newtheorem{lemma}[theorem]{Lemma}
	\newtheorem{proposition}[theorem]{Proposition}
	\newtheorem{corollary}[theorem]{Corollary}
	\newtheorem*{t2}{Theorem 1}
	\newtheorem*{p1}{Theorem 2}
	
	\theoremstyle{definition}
	\newtheorem{definition}[theorem]{Definition}
	\newtheorem{example}[theorem]{Example}
	\newtheorem{conjecture}[theorem]{Conjecture}
	\newtheorem{remark}[theorem]{Remark}

	\title{On the p-rank of curves}
	\author{SADIK TERZ\.{I}}
	
	\address{Middle East Technical University (NCC), Mathematics Department, Kalkanli, KKTC, Mersin 10, Turkey.}
	\email{terzi@metu.edu.tr}
	
	%\date{\today}
	
	
	
	\maketitle
	
	\begin{abstract}
		In this paper, we are concerned with the computations of the $p$-rank of curves in two different setups. We first work with complete intersection varieties in $\mb{P}^n \text{ for}~n\ge 2$ and compute explicitly the action of Frobenius on the top cohomology group. In case of curves and surfaces, this information suffices to determine if the variety is ordinary. Next, we consider curves on more general surfaces with $p_g(S) = 0 = q(S)$ such as Hirzebruch surfaces and determine $p$-rank of curves on Hirzebruch surfaces.  
	\end{abstract}
	
	\vskip 0.5cm
	\section{Introduction}
Let $X$ be a smooth projective  curve of genus $g\geq 2$ over an algebraically closed field $k$ of  characteristic  $p > 0$. The $p$-rank $\sigma(X)$ and $a$-number $a(X)$ are fundamental invariants of $X$(see Definition \ref{od} below), and have been studied extensively by determining the action of the Frobenius map on the cohomology group $H^1(X, \mc{O}_X)$ or equivalently the action of the Cartier operator on $H^0(X, \Omega_X)$. In the former case one essentially determines the Hasse-Witt matrix (\cite{M}) and in the latter the Cartier-Manin matrix (\cite{Y})  describing the action. For a more extensive bibliography on Hasse-Witt and Cartier-Manin matrices we refer to \cite{A}.  We are interested in determining the $p$-rank $\sigma (X)$ and $a$-number $a(X)$ for curves in certain varieties.
\par In the second section, we consider complete intersection curves in a projective space ${\mb P}^n$ cut out by hypersurfaces $X_i:f_i=0$ where $f_i$ is a degree $n_i$ polynomial for $i=1,\cdots,r$. Writing $Y_i=X_1\cap X_2 \cap \cdots  \cap X_i$ for $i=1,2,\ldots ,r$ and $Y_0= {\mb P}^n_k$, we identify $  H^{n-i}(Y_i,\mathcal{O}_{Y_i})$ as a subspace of $ H^n(Y_0,\mathcal{O}_{Y_0}(-\sum_{j=1}^{r}n_{j}))$ consisting of classes killed by multiplication by $f_j$ for $j=1,\cdots,i$. Theorem \ref{Thm} provides an explicit description of these classes amenable to computation: 
\begin{theorem}{\label{Thm}}
		For each $i=1,\cdots,r$, we have the following isomorphism of the vector spaces
			$$ H^{n-i}(Y_i,\mathcal{O}_{Y_i}) \cong \{\alpha \in  H^n(Y_0,\mathcal{O}_{Y_0}(-\sum_{t=1}^{r}n_{t})) \mid \alpha f_j = 0 \text{ in }  ~H^n(Y_0,\mathcal{O}_{Y_0}(-\sum_{\substack{t=1 \\ t\neq j   } }^{i} n_t) ) \text{ } \forall j=1,\ldots,i \} $$
				where $f_j$ is a homogenous polynomial of degree $n_j$ for $j= 1,2, \cdots , i$.
		\end{theorem}
In Theorem ~\ref{P13} below, we prove that the action of the Frobenius map on this subspace is given by $ F_i^*=(f_if_{i-1} \cdots f_1)^{p-1}F_0^*$ where $F_0^*$ is Frobenius on the top cohomology group of $Y_0$. We can therefore compute the action of Frobenius on $ H^{n-i}(Y_i,\mathcal{O}_{Y_i})$ by computing on $ H^n(Y_0,\mathcal{O}_{Y_0}(-\sum_{t=1}^{r}n_{t})) $, which is clearly easier.
\newpage

\begin{theorem}
    Let $$F_i^*:H^{n-i}(Y_i,\mathcal{O}_{Y_i}(m)) \longrightarrow H^{n-i}(Y_i,\mathcal{O}_{Y_i}(pm))$$
			be the Frobenius map for $m\in \mathbb{Z}$ and $i=0,1,\ldots ,r$. Then we have the following:  $$F_i^*=(f_if_{i-1} \cdots f_1)^{p-1}F_0^* \quad              \text{on  } ~ H^{n-i}(Y_i,\mathcal{O}_{Y_i}) \subseteq H^n(Y_0,\mathcal{O}_{Y_0}(-\sum_{j=0}^{r-1}n_{r-j} )) $$
   for $i=1,2,\cdots,r$.
\end{theorem}

We combine Theorem \ref{Thm} with Theorem \ref{P13} to compute the $p-$rank of complete intersection curves and to check whether complete intersection surfaces are ordinary. We work out several examples to illustrate the methods. In fact, we provide a family of examples (Ex. \ref{ge}) of complete intersection curves to find lower bound on their $a$-number by using the explicit basis constructed in Theorem \ref{Thm} and by using the explicit action of the Frobenius map computed in Theorem \ref{P13}. This family consists of generalized Fermat curves $F_{m,n}$ of type $(m,n)$, which is an $(n-2)$-dimensional family of algebraic curves in the moduli space of smooth projective curves of genus $ g(F_{m,n}) = 1 + \frac{m^{n-1}}{2}((m-1)(n-1)-2)$ (\cite{Hi}, Section 2). The investigation of algebraic curves over fields of characteristic $p>0$ is related to several problems for curves over finite fields, such as the cardinality
of the set of rational points, the search for maximal curves with respect to the Hasse-Weil
bound, properties of zeta functions and Weierstrass points on curves. Many results have been obtained for the classical Fermat curves (i.e., $n=2$) \cite{GV, VZ, N}.
\par In the third section, we provide a formula for the action of Frobenius for curves in a more general surface $S$. If  $p_g(S) = 0 = q(S)$ where $p_g(S)  \text{ and } q(S)$ are the geometric genus and the irregularity of $S$ respectively, the method is quite effective because in this case the action of Frobenius can be easily calculated. We illustrate this property in Example \ref{ex5} for curves on Hirzebruch surfaces. As there are curves not defined on the projective plane ${\mb P}^2_k$ but possibly defined on Hirzebruch surfaces, one can expect that constraining a curve in a specific ambient space and taking advantage of its geometry enables one to determine the $p$-rank and the $a$-number of the curve. In fact, the explicit computation of a basis for $H^1(X,\mathcal{O}_X) $ and the Frobenius map on $H^1(X,\mathcal{O}_X) $ will be useful to calculate $\sigma(X)$ and $a(X)$ for such curves $X$ as in section 3.
 \par Now we recall the concept of ordinarity for curves and the definitions of the $p$-rank and the $a$-number.
 \begin{definition} \label{od} \cite[Thm. 2.2]{Su}
     Set $W=H^1(X,\mathcal{O}_X)$. Let $W^{s}$ be the maximal subspace of $W$ on which the Frobenius map $F^*$ acts bijectively. The natural numbers $\sigma(X) = \operatorname{dim}_k(W^s)$ and $a(X)= \operatorname{dim}_k(\operatorname{Ker}(F^*))$ are called the $p$-rank and the $a$-number of $X$, respectively. We say that $X$ is an ordinary curve if $ W^s = W$  
		\end{definition}

  \section{Complete intersections} 
		
		In this section, we work with complete intersections of arbitrary dimension in ${\mb P}^n_k$. In order  to check if it is ordinary (Definition~\ref{ordinary}) in the case of a curve or a surface, it suffices to check the action of $F$ on only the top cohomology group (Proposition \ref{pro10}). In Theorem \ref{P13}, we explicitly compute this action. As an example, we will provide a family of examples of smooth complete intersection curves $X$ by putting lower bound on $a$-number. We recall that smooth complete intersections are generically ordinary (\cite{I1}, Theorem 2.2). 
	\par The similar computations by using different methods can be found in \cite{K2}. The writer computes the action of Frobenius map by constructing Koszul complex of graded free modules and certain special lifting homomorphisms. Our work in Theorem \ref{P13} was done independently of work in \cite{K2}. Moreover, in \cite{Mo}, the author develops new method to compute Ekedahl-Oort type of complete intersection curve $X$ by enhancing the pair consisting of $H^1(X,\mathcal{O}_X)$ with its Hasse-Witt operator $F^*$ to Hasse-Witt triple defined in \cite[Section 2.5]{Mo}. The arguments in \cite{Mo} and in this paper were done independently.
		\begin{definition}(\cite{I1}, Def. 1.1) \label{ordinary}
		 A smooth proper variety $Y$ is $\it{ordinary}$ if $H^i(Y,B^j_{Y/k})=0$ for all $i$ and $j$ where 	$B^j_{Y/k}$ is the sheaf of the locally exact $j$-th differential forms on $Y$.
		\end{definition}
		\begin{remark}
		 If our complete intersection variety $Y$ is a curve, then Definition \ref{ordinary} is equivalent saying that the Frobenius map acts bijectively on the cohomology group $H^1(Y, \mc{O}_Y)$ (which is Definition \ref{od}) by using the following short exact sequence 
		 \begin{equation}
		    0\longrightarrow \mathcal{O}_Y \xrightarrow{F_Y} {F_Y}_*\mathcal{O}_Y \longrightarrow{}B^1_{Y/k}\longrightarrow 0. \tag{2}
		 \end{equation}
		For the case of a surface, using Serre's duality between the sheaves $B^1_{Y/k}$ and $B^2_{Y/k}$ and the exact sequence (2), a surface $Y$ is ordinary if and only if the Frobenius map acts bijectively on the cohomology group $H^2(Y, \mc{O}_Y).$ 
		\end{remark}
		Let $X_i:f_i=0$ be a degree $n_i$ hypersurface in projective $n$-space
		${\mb P}^n_k$ given by homogenous polynomial $f_i$ of degree $n_i$ for
		$i=1,2,\ldots ,r$. \\ 
		Let $$Y_i=X_1\cap X_2 \cap \cdots  \cap X_i$$ 
		for $i=1,2,\ldots ,r$ and $Y_0= {\mb P}^n_k.$ Set $F_i$ as the Frobenius morphism on $Y_i$ for $i=0,\ldots,n$. For all $m\in \Z $, we have $F_i:\mathcal{O}_{Y_i}(m) \longrightarrow {F_i}_*\mathcal{O}_{Y_i}(m) = \mathcal{O}_{Y_i}(pm) $ on the sheaf level. Moreover we define $F_i^*$ as the Frobenius map on the cohomology group $H^j(Y_i,\mathcal{O}_{Y_i}(m))$ for $j=0,\ldots, n-i$, i.e,
  $$ F_i^*:H^j(Y_i,\mathcal{O}_{Y_i}(m)) \longrightarrow H^j(Y_i,\mathcal{O}_{Y_i}(pm)). $$
 The following proposition is crucial for our computations. 
		\begin{proposition}\cite[Section 78, Proposition 5]{S2}{\label{pro10}} For $i=1,2,\ldots ,r$, we have the following;
			\begin{itemize}
				\item[a)] $H^j(Y_i,\mathcal{O}_{Y_i}(m))=0$ for $m\leq 0$, $0<j<\operatorname{dim}Y_i=n-i$,
				\item[b)] $H^0(Y_i,\mathcal{O}_{Y_i}(m))=0$ for $m< 0$,
				\item[c)] $H^0(Y_i,\mathcal{O}_{Y_i})=k$.
			\end{itemize}
		\end{proposition}
		
		\vspace{0.1cm}
		
		\begin{p1}\label{P13}	  Let $$F_i^*:H^{n-i}(Y_i,\mathcal{O}_{Y_i}(m)) \longrightarrow H^{n-i}(Y_i,\mathcal{O}_{Y_i}(pm))$$
			be the Frobenius map for $m\in \mathbb{Z}$ and $i=0,1,\ldots ,r$. Then we have the following:  $$F_i^*=(f_if_{i-1} \cdots f_1)^{p-1}F_0^* \quad              \text{on  } ~ H^{n-i}(Y_i,\mathcal{O}_{Y_i}) \subseteq H^n(Y_0,\mathcal{O}_{Y_0}(-\sum_{j=0}^{r-1}n_{r-j} )) $$
		for $i=0,1,\cdots,r$.	
		\end{p1}
		
		\begin{proof}
		We first compute the action of $F_r^*$ on $H^{n-r}(Y_r,\mc{O}_{Y_r})$ by adapting the  method of (\cite{H}, Section 4.4). \\
			Consider Diagram 1  for $i=2,3,...,r$ and Diagram 2. \\
			\adjustbox{scale=0.87,left}{
				\begin{tikzcd}
				0 \arrow{r}&\mathcal{O}_{Y_{r-i}}(-\sum_{j=0}^{i-1} n_{r-j} ) \arrow[r,"\times f_{r-i+1}"]\arrow[d,"F_{r-i}"] & \mathcal{O}_{Y_{r-i}}(-\sum_{j=0}^{i-2} n_{r-j}) \arrow[r]\arrow[d,"F_{r-i}"] & \mathcal{O}_{Y_{r-i+1}}(-\sum_{j=0}^{i-2} n_{r-j})                  \arrow{r}\arrow[d,"F_{r-i}"]       &  0 \\
				0 \arrow{r} & \mathcal{O}_{Y_{r-i}}(-p\sum_{j=0}^{i-1} n_{r-j} ) \arrow[r,"\times f^p_{r-i+1}"] \arrow[d,"\times f^{p-1}_{r-i+1}"]              &  \mathcal{O}_{Y_{r-i}}(-p\sum_{j=0}^{i-2} n_{r-j})              \arrow{r}\arrow[d,"id"]        &  \mathcal{O}_{Y_{r-i}}(-p\sum_{j=0}^{i-2} n_{r-j})        \arrow{r}\arrow[d,"q"] & 0 \\
				0 \arrow{r}	&  \mathcal{O}_{Y_{r-i}}(-n_{r-i+1}-p\sum_{j=0}^{i-2} n_{r-j} ) \arrow[r,"\times f_{r-i+1}"]                            & \mathcal{O}_{Y_{r-i}}(-p\sum_{j=0}^{i-2} n_{r-j})  \arrow{r} &\mathcal{O}_{Y_{r-i}}(-p\sum_{j=0}^{i-2} n_{r-j}) \arrow{r}          &0  &
				\end{tikzcd}
			}\\
			$$\text{Diagram 1}$$
			\begin{center}
				\begin{tikzcd}
				0 \arrow{r}&\mathcal{O}_{Y_{r-1}}(-n_r ) \arrow[r,"\times f_r"]\arrow[d,"F_{r-1}"] & \mathcal{O}_{Y_{r-1}} \arrow[r]\arrow[d,"F_{r-1}"] & \mathcal{O}_{Y_r}                  \arrow{r}\arrow[d,"F_{r-1}"]       &  0 \\
				0 \arrow{r} & \mathcal{O}_{Y_{r-1}}(-pn_r ) \arrow[r,"\times f^p_r"] \arrow[d,"\times f^{p-1}_r"]              &  \mathcal{O}_{Y_{r-1}}             \arrow{r}\arrow[d,"id"]        &  \mathcal{O}_{Y^p_r}        \arrow{r}\arrow[d,"q"] & 0 \\
				0 \arrow{r}	&  \mathcal{O}_{Y_{r-1}}(-n_r ) \arrow[r,"\times f_r"]                            & \mathcal{O}_{Y_{r-1}}  \arrow{r} &\mathcal{O}_{Y_r} \arrow{r}          &0  &
				\end{tikzcd} 
			\end{center}
			\vspace{0.1cm}
			$$\text{Diagram 2}$$
			Here, $Y^p_i$ is the subscheme of $Y_{i-1}$ defined by $f^p_i=0$ for $i=i,\ldots,r$. On the other hand, $Y_i$ is a closed subscheme of $Y^p_i$ and we have the quotient map $q:\mathcal{O}_{Y^p_i} \longrightarrow \mathcal{O}_{Y_i}$ for each $i$. By Diagram 2, Proposition \ref{pro10}, and the long exact sequences corresponding to the horizontal short exact sequences in Diagram 2, we have the following \\
			\vskip 0.15cm
			\adjustbox{scale=0.95,center}{
				\begin{tikzcd}
				0 \arrow{r}&H^{n-r}(Y_r,\mathcal{O}_{Y_r}) \arrow{r}\arrow[d,"F_{r-1}^*"] &H^{n-r+1}(Y_{r-1},\mathcal{O}_{Y_{r-1}}(-n_r)) \arrow[r,"\times f_r"]\arrow[d,"F_{r-1}^*"] &H^{n-r+1}(Y_{r-1},\mathcal{O}_{Y_{r-1}})                   \arrow{r}\arrow[d,"F_{r-1}^*"]       &  0 \\
				0 \arrow{r} &H^{n-r}(Y^p_r,\mathcal{O}_{Y^p_r})  \arrow{r} \arrow[d,"q"]              & H^{n-r+1}(Y_{r-1},\mathcal{O}_{Y_{r-1}}(-pn_r))              \arrow[r,"\times f^p_r"]\arrow[d,"f^{p-1}_r"]        &H^{n-r+1}(Y_{r-1},\mathcal{O}_{Y_{r-1}})         \arrow{r}\arrow[d,"id"] & 0 \\
				0 \arrow{r}	&H^{n-r}(Y_r,\mathcal{O}_{Y_r})   \arrow{r}                           &H^{n-r+1}(Y_{r-1},\mathcal{O}_{Y_{r-1}}(-n_r))  \arrow[r,"\times f_r"] & H^{n-r+1}(Y_{r-1},\mathcal{O}_{Y_{r-1}}) \arrow{r}          &0  &
				\end{tikzcd}
			}\\
			$$\text{Diagram 3}$$
			where $\times f$ is the multiplication by $f$ map ($a \mapsto af$). From Diagram 3, we get
			\begin{equation}
			  	H^{n-r}(Y_r,\mathcal{O}_{Y_r})=\operatorname{Ker}(\times f_r) \subseteq H^{n-r+1}(Y_{r-1},\mathcal{O}_{Y_{r-1}}(-n_r)) \tag{3}   
			\end{equation}
			$$\text{and }F_r=qF_{r-1}.$$ 
			Hence, by restricting $F_r^*$ to $\operatorname{Ker}(\times f_r)$, we find
			$$F_r^*=f^{p-1}_rF_{r-1}^*.$$
			Similarly, by using the diagrams on cohomology groups deduced from Diagram 1 and inductive arguments, one can easily conclude that for $i=2,3,\ldots ,r$
			
			$$H^{n-r+i-1}(Y_{r-i+1}, \mathcal{O}_{Y_{r-i+1}}(-\sum_{j=0}^{i-2}n_{r-j})) =\text{Ker}(\times f_{r-i+1}) \subseteq H^{n-r+i}(Y_{r-i}, \mathcal{O}_{Y_{r-i}}(-\sum_{j=0}^{i-1}n_{r-j})) =\text{Ker}(\times f_{r-i})$$
			and by restricting the Frobenius map to $\operatorname{Ker}(\times f_{r-i+1}),$ we obtain $$ F_{r-i+1}^*= f^{p-1}_{r-i+1}F_{r-i}^*. $$
			\vspace{0.1cm}
			From these two observations, we obtain
			$$\text{Ker}(\times f_{r-1})\subseteq \cdots \subseteq \text{Ker}(\times f_2) \subseteq \text{Ker}(\times f_1) \subseteq H^n(Y_0,\mathcal{O}_{Y_0}(-\sum_{j=0}^{r-1}n_{r-j} ).$$
			%\begin{remark}
			Upon using expression (3), one concludes that $\text{Ker}(\times f_r) \subseteq \text{Ker}(\times f_{r-1})$.
			%\end{remark}
			Therefore, we have $$\text{Ker}(\times f_r) \subseteq \text{Ker}(\times f_{r-1})\subseteq \cdots \subseteq \text{Ker}(\times f_2) \subseteq \text{Ker}(\times f_1) \subseteq H^n(Y_0,\mathcal{O}_{Y_0}(-\sum_{j=0}^{r-1}n_{r-j} )).$$ 
		The proposition for $F_r^*$ follows $\text{by restricting }F_0^* ~~\text{to Ker}(\times f_r)   $, and employing induction. If we consider same setup for the codimension $i$ complete intersection variety $Y_i, i=1,\cdots, r-1$, we produce similar result for $F_i^*$. As a result, we prove that $ F_i^*=(f_if_{i-1} \cdots f_1)^{p-1}F_0^*$ on $ H^{n-i}(Y_i,\mathcal{O}_{Y_i})$ for $i=1,2,\cdots, r$.
  
		\end{proof}
		Next we determine explicitly $ H^{n-i}(Y_i,\mathcal{O}_{Y_i})$  as a subspace of 
		$H^n(Y_0,\mathcal{O}_{Y_0}(-\sum_{j=1}^{r}n_{j}))$. 
		\begin{t2} 
			For each $i=1,\cdots,r$, we have the following isomorphism of the vector spaces
			$$ H^{n-i}(Y_i,\mathcal{O}_{Y_i}) \cong \{\alpha \in  H^n(Y_0,\mathcal{O}_{Y_0}(-\sum_{t=1}^{r}n_{t})) \mid \alpha f_j = 0 \text{ in }  ~H^n(Y_0,\mathcal{O}_{Y_0}(-\sum_{\substack{t=1 \\ t\neq j   } }^{i} n_t) ) \text{ } \forall j=1,\ldots,i \} $$
				where $f_j$ is the homogenous polynomial of degree $n_j$ for $j= 1,2, \cdots , i$.
		\end{t2}
	
	\begin{proof}
	Define $k-$vector spaces $W$ and $W_i$ for $i=1,2,\cdots , r$ as follows:
	$$W = H^n(Y_0,\mathcal{O}_{Y_0}(-\sum_{j=1}^{r}n_{j})) \text{ and } W_i = H^n( Y_0,\mathcal{O}_{Y_0}(-\sum_{\substack{j=1 \\ j\neq i   } }^{r} n_j) ) \text{ } \forall i=1,\ldots,r \} $$ $ \text{and maps }\phi_i:W \longrightarrow W_i \text{ by } \alpha \mapsto \alpha  f_i.$ We will show that $$\operatorname{Ker}(\times f_r) \cong \{ \alpha \in W \mid \alpha f_i = 0 
	\text{ in }  ~W_i \text{ for } i=1,2,\cdots , r\} = \displaystyle \bigcap_{i=1}^{r}\operatorname{Ker}(\phi_i) $$ by inducting on the statement that the vector space 
		 $$ \{ \alpha \in W \mid \alpha f_j = 0 \quad \text{in }  ~W_j \text{ for } j=1,2,\cdots , i\} =  \displaystyle \bigcap_{j=1}^{i}\operatorname{Ker}(\phi_j) $$ is isomorphic to the vector space $\operatorname{Ker}(\times f_i)$, described in Theorem \ref{P13}. \\ Consider the following short exact sequence 
			$$0 \longrightarrow \operatorname{Ker}(\times f_1) \longrightarrow H^n(Y_0,\mathcal{O}_{Y_0}(-\sum_{j=1}^{r}n_{j})) \xrightarrow[]{\times f_1} H^n(Y_0,\mathcal{O}_{Y_0}(-\sum_{j=2}^{r}n_{j}))\longrightarrow 0.$$
			
		By the above short exact sequence, we obtain the statement for $i=1$, i.e.,
			 $$\operatorname{Ker}(\times f_1)= \{ \alpha \in W \mid \alpha f_1 = 0 
	\text{ in }  ~W_1 \} = \operatorname{Ker}(\phi_1). $$ 
Now we assume that our statement is true for $i=1,2,\cdots ,r-1$ and prove it for $i=r$. We note that 	$$ \text{Ker}(\times f_{r-1})=H^{n-r+1}(Y_{r-1},\mathcal{O}_{Y_{r-1}}(-n_r)) \subseteq W= H^n(Y_0,\mathcal{O}_{Y_0}(-\sum_{j=1}^{r}n_{j}))  $$ and we have the short exact sequence on the cohomology groups 
					$$0 \longrightarrow \operatorname{Ker}(\times f_r) \longrightarrow  H^{n-r+1}(Y_{r-1},\mathcal{O}_{Y_{r-1}}(-n_r)) \xrightarrow[]{\times f_r} H^{n-r+1}(Y_{r-1},\mathcal{O}_{Y_{r-1}}) \longrightarrow 0$$ 
     induced by the short exact sequence 
    \begin{center}
      \begin{tikzcd}
         0 \arrow{r}&\mathcal{O}_{Y_{r-1}}(-n_r ) \arrow[r,"\times f_r"] & \mathcal{O}_{Y_{r-1}} \arrow[r] & \mathcal{O}_{Y_r}                  \arrow{r}      &  0.   
      \end{tikzcd}  
    \end{center} 
From here, we conclude that for all $\alpha \in \operatorname{Ker}(\times f_{r-1}) \text{ with } \alpha f_r =0, \text{ we have } \alpha \in \operatorname{Ker}(\times f_r). $ 
				To complete the proof we need to show only that $\alpha f_r = 0 \text{ in } W_r$ for all $\alpha \in \operatorname{Ker}(\times f_r)$.  \\
For the last statement, we will use induction argument one more time. The statement is as follows:  \\
For all $ \alpha \in \operatorname{Ker}(\times f_r)$, $\alpha \in H^{n-r+i}(Y_{r-i}, \mathcal{O}_{Y_{r-i}}(-\sum_{j=0}^{i-1}n_{r-j}))$ for $i=1,2,\ldots ,r$ and $\alpha f_r = 0 $ in \\ 
			$$ H^{n-r+i}(Y_{r-i},\mathcal{O}_{Y_{r-i}}(n_r- \sum_{j=0}^{i-1}n_{r-j})) $$ for $i=1,\ldots,r$.
We consider the following short exact sequence \\
			$$0 \longrightarrow \operatorname{Ker}(\times f_r) \longrightarrow  H^{n-r+1}(Y_{r-1},\mathcal{O}_{Y_{r-1}}(-n_r)) \xrightarrow[]{\times f_r} H^{n-r+1}(Y_{r-1},\mathcal{O}_{Y_{r-1}}) \longrightarrow 0$$  
which exhibits the above for $i=1$. Now we assume that the statement is true for $i=1,2,\ldots ,r-1$ and deduce it for $i=r$. \\
Define $$Z_{r-i+1}= X_r\cap X_1 \cap \cdots \cap X_{r-i} \subseteq Y_{r-i}$$
			for $i=2,\ldots ,r-1$ and 
			$$Z_1= X_r \subseteq Y_0.$$
			Now consider Diagram 4: \\
			
			\vspace{0.1 cm}
			\adjustbox{scale=0.80,center}{
				\begin{tikzcd}
				& 0                \arrow{d}             & 0                 \arrow{d}                 & 0                 \arrow{d}          &   \\
				0 \arrow{r} & H^{n-2}(Z_2,\mathcal{O}_{Z_2}(-\sum_{i=2}^{r-1}n_i)) \arrow{r}\arrow{d} & H^{n-1}(Y_1,\mathcal{O}_{Y_1}(-\sum_{i=2}^{r}n_i))                \arrow[r,"\times f_r"]\arrow{d}       & H^{n-1}(Y_1,\mathcal{O}_{Y_1}(-\sum_{i=2}^{r-1}n_i))          \arrow{r}\arrow{d}           & 0 \\
				0 \arrow{r} & H^{n-1}(Z_1,\mathcal{O}_{Z_1}(-\sum_{i=1}^{r-1}n_i)) \arrow{r} \arrow[d,"\times f_1"]             &  H^n(Y_0,\mathcal{O}_{Y_0}(-\sum_{i=1}^{r}n_i))   \arrow[r,"\times f_r"]\arrow[d,"\times f_1"]      & H^n(Y_0,\mathcal{O}_{Y_0}(-\sum_{i=1}^{r-1}n_i))        \arrow{r}\arrow[d,"\times f_1"] & 0 \\
				0 \arrow{r}	& H^{n-1}(Z_1,\mathcal{O}_{Z_1}(-\sum_{i=2}^{r-1}n_i)) \arrow{r} \arrow{d}                           & H^n(Y_0,\mathcal{O}_{Y_0}(-\sum_{i=2}^{r}n_i))  \arrow[r,"\times f_r"]\arrow{d} & H^n(Y_0,\mathcal{O}_{Y_0}(-\sum_{i=2}^{r-1}n_i)) \arrow{r}\arrow{d}         &0   \\
				& 0                                & 0                                           & 0                                    &
				\end{tikzcd}
			}
			$$\text{Diagram 4}$$
			By the induction hypothesis, we have
			$$\alpha \in H^{n-1}(Y_1,\mathcal{O}_{Y_1}(-\sum_{i=2}^{r}n_i))$$
   for $i=r-1$. Thus  we see that
			$$\alpha \in  H^n(Y_0,\mathcal{O}_{Y_0}(-\sum_{i=1}^{r}n_i))$$ 
			from the middle vertical short exact sequence in Diagram 4. We get 
			$$\alpha f_r = 0\text{ in } H^n(Y_0,\mathcal{O}_{Y_0}(-\sum_{i=1}^{r-1}n_i))= W_r \text{ because } \alpha f_r = 0
			\text{ in } H^{n-1}(Y_1,\mathcal{O}_{Y_1}(-\sum_{i=2}^{r-1}n_i)) $$  by observing the upper right commutative square in Diagram 4. This completes the proof for $Y_r$. If we also follow same steps as above for $ Y_i =X_1\cap X_2 \cap \cdots  \cap X_i, i=1,\cdots,r-1$, we prove that   $\alpha \in H^{n-i}(Y_i,\mathcal{O}_{Y_i})$ if and only if $\alpha f_j = 0 \text{ in }W_j$ for $j=1,\cdots,i$. This finishes the proof.
			
		\end{proof} 
		\vspace{0.1cm}
	\begin{example} \label{ex4}
			We consider the smooth curve $X$ over a field $k$ of positive characteristic $p=2$ which is a complete intersection in ${\mb P}^3_k$   given by the surfaces
			$$S_1:f=xw-yz=0$$
			$$S_2:g=y^3+z^3+w^3+\lambda x^3=0,\quad \text{where } \lambda \neq \lambda^2.$$  A more general family of examples of complete intersection curves can be found in \cite[Section 3.1, Prop. 3.1.4]{K2}.
Since the vector spaces  $H^3({\mb P}^3,\mathcal{O}_{{\mb P}^3}(-3))=0$ and $H^3({\mb P}^3,\mathcal{O}_{{\mb P}^3}(-2)) = 0$, the maps $$[\times f] : H^3({\mb P}^3,\mathcal{O}_{{\mb P}^3}(-5)) \to H^3({\mb P}^3,\mathcal{O}_{{\mb P}^3}(-3)),$$
			$$[\times g] : H^3({\mb P}^3,\mathcal{O}_{{\mb P}^3}(-5)) \to H^3({\mb P}^3,\mathcal{O}_{{\mb P}^3}(-2))$$ are both zero.
Therefore, we have $$H^1(X,\mathcal{O}_X) \cong H^3({\mb P}^3,\mathcal{O}_{{\mb P}^3}(-5)),$$
			$$ F:=F_X^*=(fg)^{p-1}F^*_{{\mb P}^3}=fgF^*_{{\mb P}^3}$$
			and $$ B=\{\alpha_1 = \dfrac{1}{x^2yzw}, \alpha_2 = \dfrac{1}{xy^2zw}, \alpha_3 = \dfrac{1}{xyz^2w}, \alpha_4 = \dfrac{1}{xyzw^2}  \}$$ is a basis for $H^1(X,\mathcal{O}_X)$. It follows that
   \begin{align*}
F(\alpha_1) &= (fg)^{p-1}\alpha_1^p \\ &= fg\alpha_1^2 \\ &= \mathbin{\phantom-} (xy^3w +xz^3w+ xw^4+\lambda x^4w \\
  &- y^4z-yz^4-yzw^3-\lambda x^3yz )\frac{1}{x^4y^2z^2w^2} \\
   &= \mathbin{\phantom-} \frac{xy^3w}{x^4y^2z^2w^2} + \frac{xz^3w}{x^4y^2z^2w^2}+ \frac{xw^4}{x^4y^2z^2w^2} +\frac{\lambda x^4w}{x^4y^2z^2w^2 } \\
  &- \frac{y^4z}{x^4y^2z^2w^2} - \frac{yz^4}{x^4y^2z^2w^2} - \frac{yzw^3}{x^4y^2z^2w^2}- \frac{\lambda x^3yz}{x^4y^2z^2w^2}\\
  & = \frac{\lambda}{xyzw^2} \\ &=\lambda \alpha_4
 \end{align*}
in $H^1(X,\mathcal{O}_X)$ and similarly by applying $F$ to other basis elements we get
			$$F(\alpha_2)=fg\alpha_2^2= \alpha_3, $$ 
   $$F(\alpha_3)=fg\alpha_3^2=  \alpha_2,$$ $$F(\alpha_4)=fg\alpha_4^2= \alpha_1. $$
			This implies that $F$ is bijection on $H^1(X,\mathcal{O}_X)$. Thus, $X$ is an ordinary curve.

		\end{example}	
We will end this section by giving a family of smooth complete intersection curves in ${\mb P}^n$. We will obtain a lower bound on the $a$-number of the curves in this family by using  the action of Frobenius on cohomology. We will use the following example \cite[Section 2.2]{Hi} of smooth integral complete intersection curves :
 \begin{example} \label{ge}[Generalized Fermat Curve]
     Let $X$ be the curve defined as follows:
     
	$$ C^m(\lambda_0,\lambda_1,...,\lambda_{n-2}) := \begin{Bmatrix}  \lambda_0 x_0^m+x^m_1+x^m_2=0\\ \lambda_1 x_0^m+x^m_1+x^m_3=0 \\ \vdots \\ \lambda_{n-2}x_0^m+x^m_1+x^m_n=0\end{Bmatrix} \subset {\mb P}^n $$
 where $\lambda_0,\lambda_1,...,\lambda_{n-2}$ are pairwise different elements of field $k$ with $\lambda_i \neq 0 \text{ for } i=0,1,...,n-2$. We set $f_i = \lambda_ix_0^m+x^m_1+x^m_{i+2} $ for $i=0,1,...,n-2$. In Theorem \ref{Thm} we have described elements of  $H^1(X,\mathcal{O}_X)$ as follows:
 $$ \{\alpha \in  H^n({\mb P}^n,\mathcal{O}_{{\mb P}^n}(-(n-1)m)) \mid \alpha f_i = 0 \quad \text{in }  ~H^n({\mb P}^n,\mathcal{O}_{{\mb P}^n}(-(n-2)m) ) \text{ for } i=0,1,\ldots,n-2 \}.$$
   $$ H^n({\mb P}^n,\mathcal{O}_{{\mb P}^n}(-s) = \operatorname{Span}_k(\{ \dfrac{1}{ x_0^{\alpha_0}x_1^{\alpha_1} \cdots x_n^{\alpha_n}} \mid \displaystyle\sum_{i=0}^{n} \alpha_i = s, \text{ } \alpha_i \geq 1 \} )$$ for $ s\geq n-2$ \cite[Chapter III, Thm. 5.1]{H}. \par For 
a cohomology class $x_0^{-\alpha_0}x_1^{-\alpha_1}\cdots x_n^{-\alpha_n} \in H^n({\mb P}^n,\mathcal{O}_{{\mb P}^n}(-(n-1)m))$, we have \begin{align*}
 x_0^{-\alpha_0}x_1^{-\alpha_1} \cdots x_n^{-\alpha_n}f_i &= \mathbin{\phantom+} \lambda_i\,x_0^{-\alpha_0+m}\,x_1^{-\alpha_1} \, \cdots \, x_n^{-\alpha_n} \\ &+ \, x_0^{-\alpha_0} \, x_1^{-\alpha_1+m} \,  \cdots \, x_n^{-\alpha_n}  \\
  &+ \, x_0^{-\alpha_0} \, \cdots \,  x_{i+1}^{-\alpha_{i+1}}\, x_{i+2}^{-\alpha_{i+2}+m}\,  x_{i+3}^{-\alpha_{i+3}} \,  \cdots \,  x_n^{-\alpha_n}  
\end{align*} 
in $ H^n({\mb P}^n,\mathcal{O}_{{\mb P}^n}(-(n-2)m)) $ for all $i=0,\cdots,n-2$.
 We define sets $S_i$ and $S$ as $$S_i=\{ (\alpha_0,\alpha_1,\ldots,\alpha_n) \in \mathbb{N}^{n+1} \mid x_0^{-\alpha_0}x_1^{-\alpha_1} \cdots x_n^{-\alpha_n}f_i=0, \text{ } \displaystyle\sum_{j=0}^{n} \alpha_j = (n-1)m \text{ and } \alpha_j \geq 1 \} $$ for $i=0,1,\ldots , n-2$ and
$S= \displaystyle\bigcap^{n-2}_{i=0}S_i$.  We now see that 
\begin{align*}
 x_0^{-\alpha_0}x_1^{-\alpha_1} \cdots x_n^{-\alpha_n}f_i &= \mathbin{\phantom+} \lambda_i\,x_0^{-\alpha_0+m}\,x_1^{-\alpha_1} \, \cdots \, x_n^{-\alpha_n} \, + \, x_0^{-\alpha_0} \, x_1^{-\alpha_1+m} \,  \cdots \, x_n^{-\alpha_n}  \\
  &+ x_0^{-\alpha_0} \, \cdots \, x_{i+2}^{-\alpha_{i+2}+m} \,  \cdots \,  x_n^{-\alpha_n} \\
  & = 0
 \end{align*}
if and only if $ -\alpha_{i+2}+m \geq 0, \text{ } -\alpha_0+m \geq 0 \text{ and } -\alpha_1+m \geq 0$ for $i=0,1, \ldots , n-2$ if and only if $ \alpha_i \leq m $ for $i= 0,1, \ldots , n$. Therefore, we have $$S= \{ (\alpha_0,\alpha_1,\cdots,\alpha_n) \in \mathbb{N}^{n+1} \mid   \displaystyle\sum_{j=0}^{n} \alpha_j = (n-1)m \text{ and } 1\leq \alpha_i \leq m \text{ for } i=0,1, \ldots ,n\}.$$ Note that $ \operatorname{Span}_k( \{ x_0^{-\alpha_0}x_1^{-\alpha_1} \cdots x_n^{-\alpha_n} | (\alpha_0, \cdots, \alpha_n) \in S \} ) \subsetneq H^1(X,\mathcal{O}_X)$. Let us compute the cardinality $|S|$ of $S$. We are looking for non-negative integer solutions of the problem: \\
 $$ \left\{ 
    \begin{array}{lr}
        \displaystyle\sum_{j=0}^{n} \alpha_j = (n-1)m-(n+1) \\
        \alpha_j \leq m-1 \text{ for } j=0,\ldots ,n
    \end{array}
\right\}  $$
 Let $N(0)$ be the number of all non-negative integer solutions of $ \displaystyle\sum_{j=0}^{n} \alpha_j = (n-1)m-(n+1) $ and $N(i)$ be the number of non-negative integer solutions of $ \displaystyle\sum_{j=0}^{n} \alpha_j = (n-1)m-(n+1) $ such that at least $i$ of $ \alpha_0, \ldots, \alpha_n$ is greater than or equal to $m$. Then by the principle of Inclusion-Exclusion we find $ |S|= \displaystyle \sum_{i=0}^{n} (-1)^iN(i)$ where 
 $$ N(i) = \binom{n+1}{i} \operatorname{Card} \left\{ 
    \begin{array}{lr}
        (\alpha_0,\ldots, \alpha_n):\displaystyle\sum_{j=0}^{n} \alpha_j = (n-1)m-(n+1)-im
    \end{array}
\right\} $$ for $i=0,\ldots,n$. Hence $N(i)= \binom{n+1}{i} \binom{(n-i-1)m-1}{n}$ for $i=0,\ldots,n$.
 \vskip 0.5cm
 
 We now assume that
 $  \text{char}(k) = 2 \text{ and } m\geq 3 \text{ is an odd integer} $ and we will compute the Frobenius map $F^*$ on the set $\{ x_0^{-\alpha_0}x_1^{-\alpha_1} \cdots x_n^{-\alpha_n} | (\alpha_0, \ldots, \alpha_n) \in S \}$ . For $(\alpha_0, \ldots, \alpha_n) \in S$,
 \begin{equation} \label{eq2}
\begin{split}
	F^*( x_0^{-\alpha_0}\cdots x_n^{-\alpha_n}) & =  (f_0\cdots f_{n-2}) x_0^{-2\alpha_0} \cdots x_n^{-2\alpha_n} \\
 & =  \sum_{\substack{\rho \in \operatorname{Sym}(\{0,\ldots,n \}) \\ \rho = (\rho_0\rho_1)(\rho_2\cdots \rho_n)  \\ \text{the sequence } \left\{ \rho_i \right\}^{n}_{i=2}\\ \text{ decreases at most twice}   } } h_{\rho}  x_{\rho_0}^{-2\alpha_{\rho_0}}x_{\rho_1}^{-2\alpha_{\rho_1}} x_{\rho_2}^{-2\alpha_{\rho_2}+m} \cdots x_{\rho_n}^{-2\alpha_{\rho_n}+m}
	\end{split}
	\end{equation}

in $H^1(X,\mathcal{O}_X)$ where $h_{\rho} = h_{\rho}(\lambda_0, \ldots , \lambda_{n-2}) \neq 0 $ as \[ h_{\rho}= \begin{cases} 
      \lambda_i & \text{ if } \rho_i = 0 \text{ for some }i\geq2 , \\
      1 & \text{ otherwise}.
   \end{cases}
\] Let us define sets $S_{\rho}= \{\rho_2, \cdots, \rho_n \}$ for each $h_{\rho}$ in the eqn. \ref{eq2}. Note that if $S_\rho = S_{\rho'} = S_{\rho''} $, then we have either $\rho=\rho'$ or $\rho=\rho''$ because the sequence $ \left\{ \rho_i \right\}^{n}_{i=2}$ decreases at most twice. If we sum up the coefficients of same (Laurent) monomial, we see that the sum $$\displaystyle\sum_{\rho}^{}h_{\rho}  x_{\rho_0}^{-2\alpha_{\rho_0}}x_{\rho_1}^{-2\alpha_{\rho_1}} x_{\rho_2}^{-2\alpha_{\rho_2}+m} \cdots x_{\rho_n}^{-2\alpha_{\rho_n}+m} $$ becomes $\sum_{}^{}a_Ix^I$ where \[ a_I= \begin{cases} 
      h_{\rho}+ h_{\rho'} & \text{ if } \rho_i = 0 =\rho'_j \text{ for some }i,j\geq2 \text{ with } i\neq j \text{ and } S_\rho = S_{\rho'}, \\ h_{\rho} & \text{ if } \rho_i = 0 \text{ for some }i\geq2 \text{ and }  S_\rho \neq S_{\rho'} \text{ for any } \rho' \neq \rho,
      \\  1 & \text{ otherwise}.
   \end{cases}
\]
 As $a_I \neq 0$, there is no cancellation in the sum \ref{eq2}. We observe that $F^*( x_0^{-\alpha_0}\cdots x_n^{-\alpha_n})= 0 $ if and only if at least one of the terms $-2\alpha_{\rho_i} +m \geq 0 $ for $\rho$ and $i=2,\ldots, n$ in each summation of the sum (\ref{eq2}), if and only if at least one of the terms $ \alpha_{\rho_i} \leq \frac{m-1}{2}$ ($m$ is odd) for $\rho$ and $i=2,\ldots, n$ in each summation of the sum (\ref{eq2}). This is the case when at least three of  $ \alpha_{i}$ are less than or equal to $\frac{m-1}{2}$. Let $T$ be the positive integer solutions of this problem, i.e., solutions of the following:
$$ \left\{ 
    \begin{array}{lr}
        \displaystyle\sum_{j=0}^{n} \alpha_j = (n-1)m\\
       1\leq \alpha_j \leq m \text{ for } j=0,\ldots ,n \\
        1\leq \alpha_{j_1},\ldots \alpha_{j_s} \leq (m-1)/2 \text{ for } s\geq 3
    \end{array}
\right\}  $$
Then $T = |S|-\binom{n+1}{n-1}T(n-1) + \binom{n+1}{n}T(n)- \binom{n+1}{n+1}T(n+1) $ where $T(i)$ is the number of non-negative integer solutions of $ \displaystyle\sum_{j=0}^{n} \alpha_j = (n-1)m-(n+1) $ such that at least $i$ of $ \alpha_0, \ldots, \alpha_n$ is greater than or equal to $(m-1)/2$ for $i\geq n-1$. Therefore 
$$ T(i) = \binom{n+1}{i} \operatorname{Card} \left\{ 
    \begin{array}{lr}
        (\alpha_0,\ldots, \alpha_n):\displaystyle\sum_{j=0}^{n} \alpha_j = (n-1)m-(n+1)-i(m-1)/2
    \end{array}
\right\}.$$ Hence $T(i)= \binom{n+1}{i} \binom{(n-1)m-i(m-1)/2-1}{n}$ for $i\geq n-1$. 
\vskip 0.5cm
As a result we have the following inequality:
$$a(X) = \operatorname{Rank}(\operatorname{Ker}(F^*)) \geq T.$$
 \end{example}
\section{Curves in Hirzebruch surfaces} 

In this section, we want to compute the action of Frobenius for curves on more general surfaces. Let $S$ be a smooth projective surface over an algebraically closed field of positive characteristic $p$ with invariants geometric genus $p_g=0$ and irregularity $q=0$. Let $X$ be a projective curve on $S$ with corresponding divisor $D$. We have the following short exact sequence which defines our curve. 
			$$0\longrightarrow \mathcal{O}_S(-D) \longrightarrow \mathcal{O}_S \longrightarrow \mathcal{O}_X \longrightarrow 0.$$ 
			By using the long exact sequence of cohomology obtained from the above short exact sequence, one sees that 
			$$H^1(X,\mathcal{O}_X) \cong H^2(S, \mathcal{O}_S(-D)) $$

		As an application, we illustrate how to compute the $p$-rank of integral curves on Hirzebruch surfaces
		$ \mathscr{H}_r$.
		\begin{example} \label{ex5}
			Let $k$ be an algebraically closed field of positive characteristic $p$. We will review the construction of the $r$-th Hirzebruch surface $\mathscr{H}_r$. Let us consider the fan (\cite{C}, Ex.3.1.16) 
   $$\Sigma_r = \{ \sigma_1 , \sigma_2   , \sigma_3, \sigma_4  , \rho_1 = \sigma_3 \cap \sigma_4 , \rho_2 = \sigma_1 \cap \sigma_4, \rho_3 = \sigma_1 \cap \sigma_2,  \rho_4 = \sigma_2 \cap \sigma_3, 0  \} $$ 
   in $\mathbb{R}^2$ where 
   $$ \sigma_1 = \operatorname{Cone}(e_1,e_2),\sigma_2 = \operatorname{Cone}(e_1, -e_2), \sigma_3= \operatorname{Cone}(-e_1+re_2,-e_2) , \sigma_4 = \operatorname{Cone}(-e_1+re_2,e_2). $$
   The corresponding toric variety $X_{\Sigma_r} $ is covered by open affine subsets,

    $$ \left\{ 
    \begin{array}{lr}
       U_{\sigma_1} = \operatorname{Spec}(k[x,y]) \cong k^2 \\ U_{\sigma_2} = \operatorname{Spec}(k[x,y^{-1}]) \cong k^2 \\ U_{\sigma_3} = \operatorname{Spec}(k[x^{-1},x^{-r}y^{-1}]) \cong k^2 \\ U_{\sigma_4} = \operatorname{Spec}(k[x^{-1},x^ry]) \cong k^2 
    \end{array} \right\}  $$ and glued according to (\cite{C} Prop.3.1.3). We call  $X_{\Sigma_r} $ the Hirzebruch surface $\mathscr{H}_r$. We set 
 $$  u_{\sigma_1} = -e_1 + re_2, \text{ } u_{\sigma_2} = e_2, \text{ }  u_{\sigma_3} = e_1,  \text{ }  u_{\sigma_4} = -e_2  $$ for the ray generators (\cite{C}, Lemma 1.2.15) of the one dimensional cones in $\Sigma_r$. Moreover the total coordinate ring of $\mathscr{H}_r$ is $$ R= k[x_{\sigma} \mid \sigma \in \Sigma(1)] $$
where $ \Sigma(1)$ is the set of all 1-dimensional cones in $ \Sigma_r$. Here we label $x_{\sigma_i}$ as $x_i$ and hence $ R= k[x_1, x_2,x_3, x_4 ]$. Now we will describe how to put grading on the ring $S$(\cite{C}, Section 5.2): \\
Let $M$ be the lattice of characters of the torus $(\mathbb{C}^*)^2$ of the surface $\mathscr{H}_r$. Note that $M$ is isomorphic to $(\mathbb{Z})^2$. For given $m\in M$, we define the principal divisor $\operatorname{div}(\chi^m)$ as 
$$ \operatorname{div}(\chi^m)= \displaystyle\sum_{\rho}^{} \langle m, u_{\rho} \rangle D_{\rho}$$ where $\langle, \rangle$ is the usual dot product in $\mathbb{Z}^2$. Now set $D_{\rho_i} = D_i$ and $u_{\rho_i}= u_i$. One may compute the divisor class group $\operatorname{CI}(\mathscr{H}_r)$(\cite{C} Ex.4.1.8 ) as follows:
  $$ \left\{ 
    \begin{array}{lr}
      0 \sim \operatorname{div}(\chi^{e_1}) = \displaystyle\sum_{i=1}^{4}\langle e_1,u_i\rangle D_i = -D_1 + D_3 \text{ which implies }  D_1 \sim D_3 \\  0 \sim \operatorname{div}(\chi^{e_2}) = \displaystyle\sum_{i=1}^{4}\langle e_2,u_i\rangle D_i = rD_1 + D_2 - D_4 \text{ which implies }  D_2 \sim -rD_3 + D_4
    \end{array} \right\}  $$
 Therefore we see that $\operatorname{CI}(\mathscr{H}_r)$ is free abelian group of rank two generated by classes $D_3$ and $D_4$ and hence $\operatorname{CI}(\mathscr{H}_r)$ is isomorphic to $\mathbb{Z}^2$. One identifies classes of $[D_3]$ and $[D_4]$ by $(1,0)$ and $(0,1)$, respectively. We have the short exact sequence 
 $$  0 \longrightarrow M \longrightarrow \mathbb{Z}^{\Sigma(1)}\longrightarrow \operatorname{CI}(\mathscr{H}_r) \longrightarrow 0 $$  where $m\in M$ is mapped to $\operatorname{div}(\chi^m)$ and $(a_{\rho})_{\rho \in \Sigma(1)}$ is mapped to $ [\displaystyle \sum_{\rho}^{} a_{\rho}D_{\rho}]$ (\cite{C} Thm.4.1.3). One defines the degree of a monomial $ x^a= \displaystyle \prod_{\rho}^{}x_{\rho}^{a_{\rho}}$ as to be $$ \operatorname{deg}(x^a) = \displaystyle [\displaystyle\sum_{\rho}^{} a_{\rho}D_{\rho}] \in \operatorname{CI}(\mathscr{H}_r).$$ 
 Therefore we obtain the followings:
 $$ \left\{ 
    \begin{array}{lr}
      \operatorname{deg}(x_1)= [D_1] = (1,0), \text{ }   \operatorname{deg}(x_2)= [D_2] = (-r,1)\\ \operatorname{deg}(x_3)= [D_3] = (1,0), \text{ }   \operatorname{deg}(x_4)= [D_4] = (0,1)
    \end{array} \right\}.  $$
    Our next object is to homogenize characters of $M$ described in (\cite{C}, Section 5.2) to determine cohomology groups on which the Frobenius map acts. In this way we will provide formula for the Frobenius map on the cohomology groups of curves in the Hirzebruch surface $ \mathscr{H}_r$. A Weil divisor $D= \displaystyle \sum_{i=1}^{4}a_iD_i$ on $\mathscr{H}_r $ yields the polyhedron 
    $$ P_D= \{ m \in M_{\mathbb{R}} = \mathbb{R}^2 \mid \langle m,u_i \rangle \geq -a_i, \text{ for } 1\leq i \leq 4\}.$$ 
   The $D$-$\textbf{homogenization}$ of $\chi^m$ is defined as to be the Laurent monomial $$x^{\langle m,D\rangle} = \displaystyle \prod_{i=1}^{4} x_i^{\langle m,D \rangle + a_i} . $$

     The cohomology group $H^0(\mathscr{H}_r, \mathcal{O}_{\mathscr{H}_r }(D)  )$ is spanned by the characters coming from lattice points of $P_D$, i.e.,
     \begin{equation} \label{eq1}
\begin{split}
H^0(\mathscr{H}_r, \mathcal{O}_{\mathscr{H}_r }(D)  ) & = \displaystyle\bigoplus_{m\in P_D \cap M}^{} k \cdot \chi^m \\
 & = \operatorname{Span}_k(\{ x_1^{\alpha_1}x_2^{\alpha_2}x_3^{\alpha_3}x_4^{\alpha_4}  \mid \displaystyle\sum_{i=1}^{4} \alpha_i\operatorname{deg}(x_i) = (a,b) \text{, } \alpha_i \geq 0 \})
\end{split}
\end{equation}
   which gives the degree $\operatorname{deg}(D)=(a,b)$ part of the ring $R$, and denoted by $R(a,b)$.
			Let $X$ be an integral curve corresponding to $D= a_1D_1 +a_2D_2 + a_3D_3 + a_4D_4  \sim aD_3+bD_4$ where $a,b>0$ and $\mathcal{O}_S(a,b)$ be the invertible sheaf given by $D$. More details can be found in (\cite{C}, Examples 6.1.16 and 6.3.23). \\
			Then  $$H^1(X,\mathcal{O}_X) \cong 
   H^2(\mathscr{H}_r, \mathcal{O}_{\mathscr{H}_r }(-a,-b) ) $$ as $p_g(\mathscr{H}_r) =0= q(\mathscr{H}_r)$. However, by using (\cite{C}, Theorem 9.2.7), we get 	$$ H^2(\mathscr{H}_r, \mathcal{O}_{\mathscr{H}_r }(-a,-b) ) = \displaystyle\bigoplus_{m\in \operatorname{Relint}(P_D)}^{} k\cdot \chi^{-m}$$ where $\operatorname{Relint}(P_D)$ is the interior of $P_D$ in $\mathbb{R}^2$. Note that $m\in \operatorname{Relint}(P_D)$ if and only if $ \langle m,u_i \rangle > -a_i$ if and only if $\langle m,u_i \rangle +a_i =b_i \geq 1 $ for $i=1,2,3,4$. Thus, $$H^1(X,\mathcal{O}_X)= \operatorname{Span}_k(
   \{  x_1^{-b_1}x_2^{-b_2}x_3^{-b_3}x_4^{-b_4}  \mid \displaystyle\sum_{i=1}^{4} b_i\operatorname{deg}(x_i) = (a,b)\text{ }, b_i \geq 1 \}).$$ Let $f \in R(a,b)$ be the polynomial defining the curve $X$ (\cite{C}, Proposition 5.2.4 and \cite{H}, Proposition 1.12A). We use the following commutative diagram to compute the action of Frobenius map $ F_X^*$  on the cohomology group $  H^1(X,\mathcal{O}_X)$.
   \begin{center}
				\begin{tikzcd}
				0 \arrow{r}&\mathcal{O}_{\mathcal{H}_r}(-a,-b ) \arrow[r,"\times f"]\arrow[d,"F_{\mathcal{H}_r}"] & \mathcal{O}_{\mathcal{H}_r} \arrow[r]\arrow[d,"F_{\mathcal{H}_r }"] & \mathcal{O}_{X}                  \arrow{r}\arrow[d,"F_{\mathcal{H}_r}"]       &  0 \\
				0 \arrow{r} & \mathcal{O}_{\mathcal{H}_r}(-pa,-pb) \arrow[r,"\times f^p"] \arrow[d,"\times f^{p-1}"]              &  \mathcal{O}_{{\mathcal{H}_r}}             \arrow{r}\arrow[d,"id"]        &  \mathcal{O}_{X^p}        \arrow{r}\arrow[d,"q"] & 0 \\
				0 \arrow{r}	&  \mathcal{O}_{ \mathcal{H}_r}(-a,-b ) \arrow[r,"\times f"]                            & \mathcal{O}_{\mathcal{H}_r}  \arrow{r} &\mathcal{O}_X \arrow{r}          &0  &
				\end{tikzcd} 
			\end{center}
			\vspace{0.1cm}
			$$\text{Diagram 5}$$
			Here, $X^p$ is the subscheme of $\mathcal{H}_r $ defined by $f^p=0$, and $F_X, \text{ }F_{ \mathcal{H}_r} $ are Frobenius morphisms of $ X \text{ and }  \mathcal{H}_r$, respectively. On the other hand, $X$ is a closed subscheme of $X^p$ and we have the quotient map $q:\mathcal{O}_{X^p} \longrightarrow \mathcal{O}_{X}$. By Diagram 5, we compute the Frobenius morphism $ F_X$ of $X$ as $$F_X=qF_{ \mathcal{H}_r}. $$  
   Let $F_X^* \text{ and } F_{ \mathcal{H}_r}^*$ be the following Frobenius maps 
   $$F_X^*: H^1(X,\mathcal{O}_X) \longrightarrow H^1(X,\mathcal{O}_X) $$ and $$ F_{ \mathcal{H}_r}^*:  H^2(\mathscr{H}_r, \mathcal{O}_{\mathscr{H}_r }(-a,-b) ) \longrightarrow  H^2(\mathscr{H}_r, \mathcal{O}_{\mathscr{H}_r }(-pa,-pb) )  $$ corresponding to morphisms  $F_X \text{ and } F_{ \mathcal{H}_r}$, respectively. By using the long exact sequences corresponding to the horizontal short exact sequences in Diagram 5, We compute the Frobenius map $F_X^*$ of $X$ as 
   $$ F_X^*=f^{p-1}F_{ \mathcal{H}_r}^* \text{ on } H^2( \mathcal{H}_r,\mathcal{O}_{\mathcal{H}_r }(-a,-b)) $$ under the identification $H^1(X,\mathcal{O}_X)=
   H^2(\mathscr{H}_r, \mathcal{O}_{\mathscr{H}_r }(-a,-b) ) $. After calculating explicit action of the Frobenius morphism, one can calculate the $p$-rank and the $a$-number of the curve $X$ by using Definition \ref{od}. 
   
   \vspace{0.1cm}
   
   To be more concrete, we consider a smooth curve $X$ over the algebraically closed field $k$ of characteristic $p=2$ which is the complete intersection in ${\mb P}^3_k$   given by the surfaces
			$$S_1:xw-yz=0$$
			$$S_2:y^3+z^3+w^3+\lambda x^3=0,\quad \text{where } \lambda \neq \lambda^2.$$ Note that this is an alternative approach to reproducing the results of Example \ref{ex4}. The surface $S_1$ is isomorphic to $\mathscr{H}_0 $ and the curve $X \subseteq S_1$ corresponds to the equation  $f=x_1^3(x_4^3 + \lambda x_2^3) + x_3^3(x_4^3 + x_2^3)=0$. As $D_1 \sim D_3$ and $D_2 \sim -rD_3 + D_4 =D_4 \text{ }(r=0)$,  we compute the degree of $f$ as 
   $$ \operatorname{deg}(f)= 3D_3 + 3D_4.$$ Therefore, we have that $$H^1(X,\mathcal{O}_X)\cong H^2(\mathscr{H}_0,\mathcal{O}_{\mathscr{H}_0}(-3,-3))$$ and 
			$$\beta = \{\rho_1 = \dfrac{1}{x_1x_3^2x_2x_4^2}, \text{ } \rho_2 = \dfrac{1}{x_1x_3^2x_2^2x_4},\text{ } \rho_3 = \dfrac{1}{x_1^2x_3x_2x_4^2},\text{ } \rho_4 = \dfrac{1}{x_1^2x_3x_2^2x_4}  \}$$ is a basis for
			$H^1(X,\mathcal{O}_X).$ The action of $F_X^*$ on $H^1(X,\mathcal{O}_X) $ is as follows:
   $$F_X^*(\rho_1)= (x_1^3(x_4^3 + \lambda x_2^3) + x_3^3(x_4^3 + x_2^3)) (\rho_1)^2 = \dfrac{ x_1^3(x_4^3 + \lambda x_2^3)}{x_1^2x_3^4x_2^2x_4^4 } + \dfrac{ x_3^3(x_4^3 +  x_2^3)}{x_1^2x_3^4x_2^2x_4^4 } = \dfrac{1}{x_1^2x_3x_2^2x_4} = \rho_4,$$
   $$F_X^*(\rho_2)= (x_1^3(x_4^3 + \lambda x_2^3) + x_3^3(x_4^3 + x_2^3)) (\rho_2)^2 = \dfrac{ x_1^3(x_4^3 + \lambda x_2^3)}{x_1^2x_3^4x_2^4x_4^2 } + \dfrac{ x_3^3(x_4^3 +  x_2^3)}{x_1^2x_3^4x_2^4x_4^2 } = \dfrac{1}{x_1^2x_3x_2x_4^2} = \rho_3,$$
    $$F_X^*(\rho_3)= (x_1^3(x_4^3 + \lambda x_2^3) + x_3^3(x_4^3 + x_2^3)) (\rho_3)^2 = \dfrac{ x_1^3(x_4^3 + \lambda x_2^3)}{x_1^4x_3^2x_2^2x_4^4 } + \dfrac{ x_3^3(x_4^3 +  x_2^3)}{x_1^4x_3^2x_2^2x_4^4} = \dfrac{1}{x_1x_3^2x_2^2x_4} = \rho_2,$$
     $$F_X^*(\rho_4)= (x_1^3(x_4^3 + \lambda x_2^3) + x_3^3(x_4^3 + x_2^3)) (\rho_4)^2 = \dfrac{ x_1^3(x_4^3 + \lambda x_2^3)}{x_1^4x_3^2x_2^4x_4^2 } + \dfrac{ x_3^3(x_4^3 +  x_2^3)}{x_1^4x_3^2x_2^4x_4^2} = \dfrac{\lambda}{x_1x_3^2x_2x_4^2} = \lambda\rho_1.$$ Thus, $F_X^*$ is a bijection on $H^1(X,\mathcal{O}_X)$ and we see that $X$ is an ordinary curve. As aresult, $a(X)=0 \text{ and } \sigma(X) =4$. 
		\end{example}


\begin{thebibliography}{}
		    \bibitem{A}
			J. D. Achter, and E. W. Howe \textit{Hasse--Witt and Cartier--Manin matrices: A warning and a request},  arXiv:1710.10726v5 [math.NT], 2020.
			
			
			\bibitem{C}
			D. A. Cox, J. B. Little, and H. K. Schenck, \textit{Toric Varieties}, American Mathematical Society, Providence, RI, 2011.
			\bibitem{GV}
   A. Garcia, J. F. Voloch, \textit{Fermat curves over finite fields}, J. Number Theory 30 no. 3 (1988), 345-356.
			\bibitem{H}
			R. Hartshorne, \textit{Algebraic Geometry}, Springer-Verlag, New York Inc. 1977.
			\bibitem{Hi}
    R. A. Hidalgo, \textit{Holomorphic differentials of generalized Fermat curves}, J. Number Theory 217 (2020), 78–101.
			\bibitem{I1}
			L.Illusie, \textit{Ordinarite des intersections compl\`{e}tes generales}, 
			Progress in Math. 87 (1990), Birkhauser-Verlag, 375-405.
			
			
   \bibitem{K1}
			M. Kudo, \textit{Computing representation matrices for the action of Frobenius to cohomology groups}, J. Symbolic Comput. 109 (2022), 441-464. 
    \bibitem{K2}
			M. Kudo, S. Harashita, \textit{Superspecial curves of genus 4 in small characteristic}, Finite Fields Appl. 45 (2017), 131-169.
				\bibitem{M}
			J. I. Manin, \textit{The Hasse–Witt matrix of an algebraic curve}, Amer. Math. Soc. Transl. (2), 45:245–264, 1965. Translated by J.W. S. Cassels.
             \bibitem{Mo}
             B. Moonen, \textit{Computing discrete invariants of of varieties in positive characteristic: I. Ekedahl-Oort types of curves}, J. Pure Appl. Algebra 226 (2022), Paper No. 107100, 19 pp.
			\bibitem{N}
               A. Neeman, \textit{Weierstrass points in characteristic $p$}, Invent. Math. 75 (1984), 359-376.
		    
			\bibitem{S2}
			J.P.Serre, \textit{Faisceaux Algebriques Coherents},  Ann. of Maths., 61, 1955, p. 197-278.
			
		
			\bibitem{Su}
			D. Surao, \textit{A formula for the Cartier operator on plane algebraic curves}, J. Reine Angew. Math. 377 (1987), 49-64. 
			\bibitem{VZ}
   J. F. Voloch, M.E. Zieve, \textit{Rational points on some Fermat curves and surfaces over finite fields}, Int. J. Number Theory 10 no. 2 (2014), 319-325.
			\bibitem{Y}
			N. Yui, \textit{On the Jacobian varieties of hyperelliptic curves over fields of characteristic p $>$ 2}, J. Algebra, 52(2):378–
            410, 1978.
			
		\end{thebibliography}
	\end{document}